\documentclass[10pt,a4paper,leqno]{amsart}
\usepackage{amssymb, amsmath, amscd}
\usepackage[dvipsnames]{xcolor}
\usepackage[hypertexnames=false]{hyperref}

\usepackage{color}

\DeclareMathOperator{\diver}{\rm div}
\DeclareMathOperator{\Hes}{\rm Hes}
\DeclareMathOperator{\Ric}{\rm Ric}
\DeclareMathOperator{\tr}{\rm tr}

\DeclareMathOperator{\spanned}{\rm span}
\newcommand{\nh}{\nabla h}
\makeatletter
\@namedef{subjclassname@2020}{%
	\textup{2020} Mathematics Subject Classification}
\makeatother

\newtheorem{theorem}{Theorem}[section]
\newtheorem{lemma}[theorem]{Lemma}

\newtheorem{corollary}[theorem]{Corollary}
\newtheorem{example}[theorem]{Example}

\theoremstyle{definition}

\theoremstyle{remark}
\newtheorem{remark}[theorem]{Remark}

\begin{document}
\title[Weighted Einstein field equations]{Vacuum Einstein field equations in smooth metric measure spaces: the isotropic case}
\author{M. Brozos-V\'azquez, D. Moj\'on-\'Alvarez}
\address{MBV: CITMAga, 15782 Santiago de Compostela, Spain}
\address{\phantom{MBV:}  
	Universidade da Coru\~na, Campus Industrial de Ferrol, Department of Mathematics, 15403 Ferrol,  Spain}
\email{miguel.brozos.vazquez@udc.gal}
\address{DMA: CITMAga, 15782 Santiago de Compostela, Spain}
\address{\phantom{DMA:}
	University of Santiago de Compostela,
15782 Santiago de Compostela, Spain}
\email{diego.mojon@rai.usc.es}
\thanks{Supported by projects PID2019-105138GB-C21/AEI/10.13039/ 501100011033 (Spain) and ED431C 2019/10 (Xunta de Galicia, Spain).}
\date{}

\maketitle

\begin{abstract} 
On a smooth metric measure spacetime $(M,g,e^{-f} dvol_g)$, we define a weighted Einstein tensor. It is given in terms of the Bakry-\'Emery Ricci tensor as a tensor which is symmetric, divergence-free, concomitant of the metric and the density function. We consider the associated vacuum weighted Einstein field equations and show that isotropic solutions have nilpotent Ricci operator. Moreover, the underlying manifold is a Brinkmann wave if it is $2$-step nilpotent and a Kundt spacetime if it is $3$-step nilpotent. More specific results are obtained in dimension $3$, where all isotropic solutions are given in local coordinates as plane waves or Kundt spacetimes. 
\end{abstract}

\section{Introduction}

Spacetimes can be generalized by introducing a density function $f$ that gives rise to a smooth metric measure space $(M,g,e^{-f} dvol_g)$. The influence of the density on the geometry of the manifold is expressed in terms of the Bakry-\'Emery Ricci tensor, which is defined as
\begin{equation}\label{eq:Bakry-Emery-Ricci-tensor}
\rho^f=\rho+\operatorname{Hes}_f-\mu df\otimes df
\end{equation}
where $\rho$ is the Ricci tensor, $\Hes_f$ is the Hessian of the function $f$ and $\mu$ is a constant. 
This tensor has been extensively studied, especially in the Riemannian setting (we refer to \cite{Lott} and references therein for some geometric properties). Although it was introduced in relation to diffusion processes \cite{Bakry-Emery}, it gave rise to the notion of quasi-Einstein manifolds (see, for example, \cite{Case-Shu-Wei,Catino2012,Catino2013} for some results in Riemannian signature and \cite{Isotropic_quasi-Einstein} in Lorentzian signature). 
 The Bakry-\'Emery Ricci tensor is an essential object in smooth metric measure spacetimes and, in a certain sense, it plays a substituting role of the usual Ricci tensor. 
 It arises in scalar–tensor gravitation
 theories, in particular when the Jordan frame is used as conformal gauge \cite{Woolgar}.
An extension of previous results to this new framework was given by Case \cite{Case}, who stated new versions of the singularity and the timelike splitting theorem in terms of this tensor.

The Einstein tensor on a spacetime $(M,g)$ is symmetric, divergence-free, concomitant of the metric tensor $g$ and its first two derivatives and linear in the second derivatives of $g$. Moreover, Lovelock \cite{Lovelock} showed that, in dimension four, these properties essentially characterize the Einstein tensor as $G=\rho+\Lambda g$, where $\Lambda$ is a constant. Our first objective is to define a tensor on a smooth metric measure space that suitably generalizes the Einstein tensor while also satisfying analogous characterizing properties.

\subsection{A weighted analogue of the Einstein tensor}
From \eqref{eq:Bakry-Emery-Ricci-tensor}, consider $\mu=1$ and the positive function $h=e^{-f}$ to rewrite a Bakry-\'Emery Ricci tensor as follows:
\[
\rho^h=\rho-\frac{\operatorname{Hes}_h}h.
\]
The particular choice $\mu=1$ is motivated by the properties we will obtain for the new tensor that we are going to define, but is also justified by geometric reasons (see  Remark~\ref{re:static} and Corollary~\ref{cor:4-warped} below). Since a generalization of the Einstein tensor must be concomitant of the metric tensor, we shall allow a summand which is a multiple of $g$. Thus, we shall consider 
a tensor of the form $\rho^h+\lambda g$, where $\lambda$ is a function on $M$. A linearization of this tensor results in
\[
G^h=h \rho-\operatorname{Hes}_h +\lambda h g.
\]
Let $\Ric$ denote the Ricci operator ($\rho(X,Y)=g(\Ric X,Y)$) and let $\tau$ denote the scalar curvature. Einstein manifolds have constant scalar curvature and we will show (see Lemma~\ref{le:const-sc} below) that the weighted analog that we are going to define also has this property. Hence, we  assume that $\tau$ is constant to compute the divergence of $G^h$:
\[
\begin{array}{rcl}
	\diver(G^h)&=& \diver(h\rho)-\diver \Hes_h+\diver (\lambda h g)\\
	\noalign{\medskip}
	&=& h \diver \rho + \iota_{\nh}\rho-d \Delta h - \iota_{\nh}\rho+ d (\lambda h)\\
	\noalign{\medskip}
	&=&  \frac12 h \, d\tau-d\Delta h +d (\lambda h)\\
	\noalign{\medskip}
	&=& d(\lambda h-\Delta h),
\end{array}
\]
where $\iota$ denotes the interior product, $\iota_X\rho=\rho(X,\cdot)$,
and we have used the contracted Bianchi identity $\diver \rho=\frac{1}2 d\tau$ and the Bochner formula $\operatorname{div}\operatorname{Hes}_h=d \Delta h+ \iota_{\nh}\rho$.
Thus, for $G^h$ to be divergence-free if $\tau$ is constant, we get that $\lambda h=\Delta h+\Lambda$, where $\Lambda$ plays the role of a cosmological constant. Consequently, we define a {\it weighted Einstein tensor} on a smooth metric measure space $(M,g,h\, dvol_g)$ by
\begin{equation}\label{eq:weighted-Einstein-tensor}
	G^h=h \rho-\operatorname{Hes}_h +(\Delta h+\Lambda) g,
\end{equation}
as a symmetric, divergence-free tensor, concomitant of the metric $g$ and the positive density $h$
and their first two derivatives. Moreover, understanding $\Lambda$ as a cosmological constant, the remaining tensor $h \rho-\operatorname{Hes}_h +\Delta hg$ is linear in the function $h$. Notice that $G^h$ is a strict generalization of the Einstein tensor, since $G^h=G$ if $h=1$. Henceforth we work in a proper smooth metric measure space, therefore $h$ is assumed to be nowhere constant so that $\nabla h\neq 0$ on any open subset.

\subsection{The vacuum weighted Einstein field equation}

From the weighted Einstein tensor, the {\it weighted Einstein field equation} is set to be $G^h=T$, where $T$ is a stress-energy tensor.  In a vacuum setting, we have $T=0$, so we define the {\it vacuum weighted Einstein field equation} as $G^h=0$, this is
\begin{equation}\label{eq:vacuum-Einstein-field-equations}
h \rho-\operatorname{Hes}_h +(\Delta h+\Lambda) g=0.
\end{equation} 
Equation \eqref{eq:vacuum-Einstein-field-equations} with $\Lambda=0$ was considered in Riemannian signature in \cite{Fischer-Marsden} from a different point of view, as it arises from the linearization of the scalar curvature function (see Remark~\ref{re:static} below).  Moreover, it was shown that, for non-constant $h$, the scalar curvature of any solution is constant. The argument extends to the Lorentzian setting and arbitrary $\Lambda$ as follows (we include details in the interest of self-containment).
\begin{lemma}\label{le:const-sc}
Let $(M,g,h\, dvol_g)$ be a smooth metric measure space that solves the vacuum weighted Einstein field equation, then the scalar curvature is constant.
\end{lemma}
\begin{proof}
	We take the divergence of equation \eqref{eq:vacuum-Einstein-field-equations} to see, using the  Bochner formula and the contracted Bianchi identity, that $0=h \diver \rho+ \iota_{\nh}\rho-\diver \Hes_h+d \Delta h=\frac12 h\, d\tau$. Hence, since $h\neq 0$ in every open subset, we conclude that $\tau$ is constant.
\end{proof}

\begin{remark}\label{re:static}
We shall point out that equation~\eqref{eq:vacuum-Einstein-field-equations} with $\Lambda=0$ is also formally related to the static perfect fluid equation (see \cite{Kobayashi,Lafontaine}), which is studied in a purely Riemannian context, since it derives from a Lorentzian situation by reducing a timelike dimension. Moreover, the same equation  appears with a different motivation in the following context. Let $L_g$ be the linearization of the scalar curvature function on a closed manifold. Its formal $L^2$-adjoint is given by $L_g^\ast f=-f \Ric_g +\Hes_f - (\Delta_g f)g$ (we refer to \cite{Besse,Bourguignon,Fischer-Marsden} for details). Considering the space of manifolds with constant scalar curvature, critical metrics for the volume functional admit non-trivial solutions for the equation $L_g^\ast f=\kappa g$ for $\kappa$ constant. This analysis was localized to the case where the metric deformation is supported on the closure of a bounded domain in \cite{Corvino-etal,Miao-Tam}, defining the $V$-static spaces.
\end{remark}

The  causal character of $\nabla h$  crucially influences the geometry of solutions to the Einstein field equation. Depending on the character of $\nabla h$ the approach in treating an equation like \eqref{eq:vacuum-Einstein-field-equations} is different, as are often distinct the features of the solutions. In this note we focus on the case in which $\nabla h$ is a lightlike vector field. Thus, we fix notation and say that a smooth metric measure space $(M,g,h\, dvol_g)$ is an {\it isotropic solution} of the vacuum weighted Einstein field equation if  \eqref{eq:vacuum-Einstein-field-equations} is satisfied and $\nabla h$ is lightlike.

\subsection{Main results}

Our main aim is to characterize isotropic solutions to the vacuum weighted Einstein field equation \eqref{eq:vacuum-Einstein-field-equations}, i.e. solutions with lightlike $\nabla h$, and describe their underlying geometric structure.
At first we consider spacetimes of arbitrary dimension $n\geq 3$. 
We will see that, in general, solutions are realized on Kundt spacetimes and, in certain cases, on Brinkmann waves. Moreover, the scalar curvature vanishes and the Ricci operator is nilpotent. We summarize the description of the geometry of the solutions in terms of the nilpotency of the Ricci operator as follows. 

\begin{theorem}\label{th:main}
	Let $(M,g,h\, dvol_g)$ be an isotropic solution of the vacuum weighted Einstein field equation. Then one of the following possibilities holds:
	\begin{enumerate}
		\item $(M,g)$ is Ricci-flat and $\Hes_h=0$.
		\item The Ricci operator is $2$-step nilpotent and $(M,g)$ is a Brinkmann wave.
		\item The Ricci operator is $3$-step nilpotent and $(M,g)$ is a Kundt spacetime.
	\end{enumerate}
\end{theorem}

In dimension three the geometry of the manifold is more rigid than in higher dimension. This implies, for example, that all Brinkmann waves that are solutions of \eqref{eq:vacuum-Einstein-field-equations} are indeed plane waves. Moreover, this rigidity allows us to describe the geometry of isotropic solutions of the vacuum weighted Einstein field equation in more detail in local coordinates, together with the explicit expression of the function $h$, as follows.

\begin{theorem}\label{th:3-dim}
	Let  $(M,g,h\, dvol_g)$ be a non-flat $3$-dimensional isotropic solution of the vacuum weighted Einstein field equation. Then, the Ricci operator is nilpotent and one of the following holds:
	\begin{enumerate}
		\item If $\Ric$ is $2$-step nilpotent then $(M,g)$ is a plane wave and there exist local coordinates $(u,v,x)$ such that 
		\begin{equation}\label{eq:plane-wave-3dim-solutions}
		g(u,v,x)=dv\left(2du-\frac{\alpha''(v)}{\alpha(v)} x^2 dv\right)+dx^2,
		\end{equation}
		where $h(u,v,x)=\alpha(v)$ is an arbitrary function with $\alpha''(v)\neq 0$.
		\item If $\Ric$ is $3$-step nilpotent then $(M,g)$ is a Kundt spacetime and there exist local coordinates $(u,v,x)$ so that $h(u,v,x)=v$ and
		\begin{equation}\label{eq:Kundt-solutions}
	g(u,v,x)=  dv( du+ F(u,v,x) dv+W(u,v,x) dx)+dx^2,
		\end{equation}
		where
		\[
		\begin{array}{rcl}
			F(u,v,x)&=&\frac{u^2}{x^2}+ \gamma_1(v,x) u+\gamma_0(v,x),\\
			\noalign{\smallskip}
			W(u,v,x)&=&-\frac{2u}{x},
		\end{array}
		\]
		with $\gamma_1(v,x)=\alpha_1(v)-\frac{2 \log (x)}{v}$ and 
		\[
		\gamma_0(v,x)=\frac{x^2 ((\log (x)-2) \log (x)+2)}{v^2}+\frac{x^2 \alpha_1(v) (1-\log (x))}{v}+x^2 \alpha_2(v)+x \alpha_3(v),
		\]
		for arbitrary functions $\alpha_1$, $\alpha_2$ and $\alpha_3$.
	\end{enumerate}
\end{theorem}

\subsection{Outline of the paper} In what follows we  will analyze the weighted Einstein field equation \eqref{eq:vacuum-Einstein-field-equations}, mainly focusing on the underlying geometric structure of isotropic solutions. We will  show that solutions are characterized by the presence of a distinguished lightlike vector  field, so we  begin by recalling some definitions of spacetimes with this property in Section~\ref{sect:distinguish-lightlike-vector-field}. In Section~\ref{sect:general-dim} we obtain the first geometric consequences of   equation \eqref{eq:vacuum-Einstein-field-equations} and prove Theorem~\ref{th:main}. Afterwards, in Section~\ref{sect:3-dim} we restrict the context to dimension three  to classify solutions on $pp$-waves, provide some illustrative examples, and prove Theorem~\ref{th:3-dim}. Finally, in Section~\ref{sect:4-dim-examples} we provide some remarks on 4-dimensional spacetimes: we prove that 4-dimensional Ricci-flat isotropic solutions are $pp$-waves; show  that the classification result in three dimensions does not extend to four dimensions by giving an appropriate example; and build Ricci-flat $4$-dimensional warped products from the solutions given in Section~\ref{sect:3-dim}.

\section{Families of spacetimes with distinguished lightlike vector field}\label{sect:distinguish-lightlike-vector-field}

When considering the vacuum weighted Einstein equation, several families characterized by the presence of a distinguished lightlike vector field  play a pivotal role. In this section we recall some definitions and basic facts about those that will appear in the subsequent analysis. 

\subsection{Kundt spacetimes}
Kundt spacetimes are interesting both from a geometrical and a physical point of view. Due to their holonomy structure, Kundt spacetimes appear in a number of physical situations. We refer to \cite{Kundt-spacetimes} for a detailed description of their geometry and to \cite{Br-Co-Her} for relations with supersymmetric solutions of supergravity theories and their role in string theory. 

We first work in arbitrary dimension $n\geq 3$. For a lightlike vector field $V$, the {\it optical scalars} of {\it expansion}, {\it shear} and {\it twist} are given, respectively, by
\begin{equation}\label{eq:optical-scalars}
	\theta=\frac{1}{n-2}\nabla_iV^i, \qquad \sigma^2=(\nabla^i V^j) \nabla_{(i} V_{j)}-(n-2)\theta^2, \qquad \omega^2=(\nabla^iV^j)\nabla_{[i}V_{j]},
\end{equation} 
where parentheses denote symmetrization and brackets denote anti-symmetrization when placed in the subindices. {\it Kundt spacetimes} are characterized by a lightlike geodesic vector field with zero optical scalars,  which means that it is expansion-free, shear-free and twist-free (see
\cite{chow2010kundt,Kundt-spacetimes,podolsky2009general}).  We also refer to \cite{MBZ} for an alternative characterization.

For an $n$-dimensional Kundt spacetime, the metric can be written in appropriate local coordinates $(u,v,x_1,\dots, x_{n-2})$  as \cite{Kundt-spacetimes,podolsky2009general} 
\begin{equation}\label{eq:local-coord-kundt-xeral}
	g=dv\left(2du+F(u,v,x)dv+\sum_{i=1}^{n-2}W_{x_i}(u,v,x)dx_i\right)+\sum_{i,j=1}^{n-2}g_{ij}(v,x)dx_idx_j,
\end{equation}
where $F$, $W_{x_i}$ and $g_{ij}$ are functions of the specified coordinates. 

In dimension three, the geometry of Kundt spacetimes is more rigid than in higher dimensions. Thus, the presence of an expansion-free lightlike geodesic vector field guarantees that the spacetime is Kundt, i.e. the vector field automatically has vanishing optical scalars \cite{chow2010kundt}. In this case, the expression \eqref{eq:local-coord-kundt-xeral} can be further normalized so that $g_{11}=1$. Thus, the metric can be written in local coordinates $(u,v,x)$ as 
\begin{equation}\label{eq:3-dim-kundt}
g(u,v,x)=  dv(2 du+ F(u,v,x) dv+W(u,v,x) dx)+dx^2.
\end{equation}

\subsection{Brinkmann waves}
A more specific situation appears when on a Kundt spacetime the distinguished lightlike geodesic vector field $V$ is recurrent, i.e. $ \nabla_X V=\omega (X)\otimes V$, for a $1$-form $\omega$. A spacetime admitting a parallel lightlike line field is said to be a {\it Brinkmann wave}. In general, if the tangent bundle admits an orthogonal direct sum decomposition into non-degenerate subspaces which are invariant under the holonomy representation, then the manifold splits as a product \cite{Wu}. However, if the holonomy representation admits an invariant subspace where the metric is degenerate and there are no proper non-degenerate invariant subspaces, then the holonomy group acts indecomposably (not irreducibly). In this case there is not such a splitting and Brinkmann waves illustrate these phenomena in Lorentzian geometry. 

Local coordinates given for Kundt spacetimes in \eqref{eq:3-dim-kundt} can be further specialized for Brinkmann waves. Thus the metric of a $3$-dimensional Brinkmann wave can be written as
\begin{equation}\label{eq:3-dim-Brinkmann-wave}
	g(u,v,x)=dv\left(2du+F(u,v,x)dv\right)+dx^2,
\end{equation}
where $V=\partial_u$ is lightlike and recurrent. Moreover, if this vector field can be rescaled to a parallel one, then $\partial_uF=0$ (see, for example, \cite{leistner}). 

\subsection{$pp$-waves and plane waves}
A special family of Brinkmann waves is that of the so-called $pp$-waves. These spacetimes appear in a number of special situations in General Relativity and, in particular, as solutions of the Einstein equations (we refer to \cite{stephani-et-al} for further details). In arbitrary dimension, $pp$-waves are Brinkmann waves which admit a parallel vector field $V$ such that $R(V^\perp,V^\perp)=0$. When particularizing to dimension three, however, the fact that $V$ is recurrent ensures the condition $R(V^\perp,V^\perp)=0$. Hence, all $3$-dimensional Brinkmann waves with parallel vector field $V$ are $pp$-waves. Thus, local special coordinates as in \eqref{eq:3-dim-Brinkmann-wave} characterize $pp$-waves if $F$ is a function of $v$ and $x$.

A $pp$-wave with transversally parallel curvature tensor (i.e. such that $\nabla_{V^\perp} R=0$) is called a {\it plane wave}. Again, we refer to \cite{stephani-et-al} for examples of contexts where these spacetimes play a role, which are numerous. In local coordinates, the metric of $3$-dimensional plane waves can be given by \eqref{eq:3-dim-Brinkmann-wave} where $F(u,v,x)= \alpha(v) x^2$. Notice that, if $\alpha$ is constant, these metrics correspond to Cahen-Wallach symmetric spaces \cite{cahen-wallach}.

\section{The vacuum Einstein field equation in arbitrary dimension}\label{sect:general-dim}

We consider a smooth metric measure space $(M,g,h\,dvol_g)$ of dimension $n$ and
 begin by analyzing the vacuum Einstein field equation. Taking traces in  \eqref{eq:vacuum-Einstein-field-equations} we have
\begin{equation}\label{eq:trace-EinsteinEquation}
0= h\tau +(n-1) \Delta h + n \Lambda, 
\end{equation}
so $\Delta h$ can be given in terms of $h$, $\tau$ and $\Lambda$ as $\Delta h=-\frac{h\tau+n \Lambda}{n-1}$. 
The following result shows that, for isotropic solutions, $\nh$ is geodesic and an eigenvector of the Ricci operator.

\begin{lemma}\label{le:geodesic-vector-field}
	Let $(M,g,h\,dvol_g)$ be an isotropic solution to the vacuum weighted Einstein field equation. Then $\nabla_{\nabla h} \nabla h=0$ and $\operatorname{Ric}(\nabla h)=\frac{h\tau+\Lambda}{(n-1)h}  \nabla h$.
\end{lemma}
\begin{proof}
	Since $g(\nabla h,\nabla h)=0$, we have
	\[
	0=(\nabla_Xg)(\nabla h,\nabla h)=-2 \Hes_h (\nabla h,X) \text{ for all vector fields } X.
	\]
	Hence $\operatorname{hes}_h(\nabla h)=\nabla_{\nabla h}\nabla h=0$ and, from equation \eqref{eq:vacuum-Einstein-field-equations}, $\operatorname{Ric}(\nabla h)=-\frac{\Delta h+\Lambda}{h} \nabla h=\frac{h\tau+\Lambda}{(n-1)h} \nabla h$.
\end{proof}

%
%

Let $\alpha=\frac{h\tau+\Lambda}{(n-1)h} $ be the eigenvalue of $\Ric$ associated to $\nabla h$. Since $\nabla h$ is lightlike and $\Ric(\nabla h)=\alpha \nabla h$, the Ricci operator has real eigenvalues. Moreover, since the Ricci operator is self-adjoint, there exists a pseudo-orthonormal basis $\mathcal{B}=\{\nh, U, E_1,\dots, E_{n-2}\}$ such that $g(\nh,U)=g(E_i,E_i)=1$ (other terms of $g$ being zero) and such that the Ricci operator satisfies $\Ric(\nabla h)=\alpha \nh$, $\Ric(U)=\nu \nh+\alpha U+\mu E_1$, $\Ric(E_1)=\mu \nh+\beta_1 E_1$ and $\Ric(E_i)=\beta_i E_i$ if $i\neq 1$ (see \cite{Oneill} for details).

In the next lemma we show that the Ricci operator is indeed nilpotent and, moreover, the constant $\Lambda$ and the Laplacian of $h$ vanish.

\begin{lemma}\label{le:Ricci-nilpotent}
	Let $(M,g,h\, dvol_g)$ be an isotropic solution of the vacuum weighted Einstein field equation. Then $\Ric$ is nilpotent, $\Delta h=0$ and $\Lambda=0$.
\end{lemma}
\begin{proof}
	By Lemma~\ref{le:const-sc}, the scalar curvature  $\tau$ is constant. We use the contracted second Bianchi identity to see that $\diver \rho(\nabla h)=\frac12 d\tau(\nabla h)=0$. Hence we have
\begin{equation}\label{eq:divric-nablah}
0=\diver \rho(\nabla h)=(\nabla_{\nabla h}\rho)(U,\nabla h)+(\nabla_{U}\rho)(\nabla h,\nabla h)+\sum_i (\nabla_{E_i}\rho)(E_i,\nabla h).
\end{equation}
We compute each of these three terms separately. Note that, since $\alpha=\frac{h\tau+\Lambda}{(n-1)h}$, we have $\nabla h(\alpha)=0$. Also, since $\nabla_{\nabla h} \nabla h=0$ and $\rho(\nabla_{\nabla h} U,\nabla h)=\alpha g(\nabla_{\nabla h} U,\nabla h)=\alpha\{\nabla h g(U,\nabla h)-g(U,\nabla_{\nabla h}\nabla h)\}=0$, we have
\[
(\nabla_{\nabla h}\rho)(U,\nabla h)=\nabla h(\rho(U,\nabla h))-\rho(\nabla_{\nabla h} U,\nabla h)-\rho(U,\nabla_{\nabla h}\nabla h)=\nabla h(\alpha)=0.
\]
Since $\rho(\nabla h,\nabla h)=0$, we see that
\[
(\nabla_U\rho)(\nabla h,\nabla h)=U(\rho(\nabla h,\nabla h))-2\rho(\nabla_U\nabla h,\nabla h)=-2\alpha g(\nabla_{\nabla h}\nabla h,U)=0.
\]
Now, since $\rho(E_i,\nabla h)=0$ for all $i$, since $\sum_i \rho(\nabla_{E_i} E_i,\nh)=-\alpha \Delta h$, and since $\sum_i \rho(E_i, \nabla_{E_i}\nh)=\tr (\Ric \circ \operatorname{hes}_h)$,  we obtain
\[
\begin{array}{rcl}
\displaystyle\sum_i (\nabla_{E_i}\rho)(E_i,\nabla h)&=& \displaystyle\sum_i \{E_i \rho(E_i,\nh)-\rho (\nabla_{E_i} E_i,\nh)-\rho(E_i, \nabla_{E_i}\nh)\}\\
\noalign{\medskip}
&=& \alpha \Delta h- \tr (\Ric \circ \operatorname{hes}_h).
\end{array}
\]
Hence, from \eqref{eq:divric-nablah} we obtain
that 
\begin{equation}\label{eq:relation-tr}
\alpha \Delta h- \tr (\Ric \circ \operatorname{hes}_h)=0.
\end{equation}
Now we set $\operatorname{hes}_h(E_1)=\star \nh+\gamma_1 E_1$ and $\operatorname{hes}_h(E_i)=\gamma_i E_i$ for $i\geq 2$. From   \eqref{eq:vacuum-Einstein-field-equations} we have
\[
\begin{array}{l}
0=G^h(\nh,U)=h\alpha+\Delta h+\Lambda,\\
0=G^h(E_i,E_i)=h \beta_i-\gamma_i+ \Delta h+\Lambda,
\end{array}
\]
so $\gamma_i=h(\beta_i-\alpha)$. Hence, from equation \eqref{eq:relation-tr} we have
\[
0=\alpha \sum_i \gamma_i-\sum_i \beta_i \gamma_i= \sum_i \gamma_i (\alpha-\beta_i)=- \sum_i \frac{\gamma_i^2}{h}. 
\]
This implies $\gamma_i=0$ for all $i$, and therefore $\Delta h=0$. Moreover, $\beta_i=\alpha$ for all $i$. Now, from \eqref{eq:trace-EinsteinEquation} we get that $h\tau+n\Lambda=0$. Since $\tau$ and $\Lambda$ are constant, but $h$ is not, we conclude $\tau=\Lambda=0$. Furthermore,  $\beta_i=\frac{h\tau+\Lambda}{(n-1)h}=0$ and $\Ric$ is nilpotent.
\end{proof}

\begin{remark}\label{re:non-isotropic}
Due to Lemma~\ref{le:Ricci-nilpotent}, if a solution to the vacuum weighted Einstein field equation is isotropic, then $\Lambda=0$. This implication does not hold if $\nabla h$ is not lightlike.

If we consider an $n$-dimensional Einstein manifold $(M,g)$, with $\rho=\frac{\tau}n g$, that satisfies equation \eqref{eq:vacuum-Einstein-field-equations}, then
\[
\Hes_h=\left(\frac{h\tau}n+\Delta h+\Lambda\right)g.
\]
Notice that solutions to this equation are necessarily solutions of the local M\"obius equation $\Hes_h=\frac{\Delta h}{n}g$ (see \cite{Osgood-Stowe,Xu}), which provide conformal changes of Einstein metrics that are also Einstein. We refer to \cite{Kuhnel-Rademacher} for a survey of this topic in pseudo-Riemannian geometry. Also, the local M\"obius equation was applied to give the warped product structure of a Schwarzschild space-time in \cite{Manolo-Eduardo-Kupeli}.

For illustrative purposes, since $3$-dimensional Einstein manifolds have constant sectional curvature, one can solve the local M\"obius equation on the de Sitter and the Anti-de-Sitter spacetimes of dimension three to provide simple examples of solutions of \eqref{eq:vacuum-Einstein-field-equations} with $\Lambda\neq 0$ as follows:
\begin{enumerate}
\item We consider de Sitter space with coordinates $(x,y,z)$ and metric 
\[
g_{dS}= \kappa^2 \left(-\cos^2 y dx^2+ dy^2+ \sin^2 y dz^2\right).
\]
The scalar curvature is given by $\tau=\frac{6}{\kappa^2}$.
A direct calculation shows that a function of the form $h(x,y,z)=-\frac{\kappa^2
	\Lambda }{2}+\sin (y) (c_1 \cos (z)+c_2 \sin (z))$ gives  solutions to the vacuum weighted Einstein equation for constants $c_1$, $c_2$. Since $\|\nabla h\|^2=\frac{1}{\kappa ^2}\left(\cos ^2(y) \left(c_2 \sin (z)+c_1 \cos (z)\right){}^2+\left(c_2 \cos (z)-c_1 \sin (z)\right){}^2\right)\geq 0$, the gradient of $h$ is spacelike or lightlike.
In conclusion, there exist local solutions of \eqref{eq:vacuum-Einstein-field-equations} for arbitrary $\Lambda$.

Moreover, notice that a conformal change of the form $h^{-2} g_{dS}$ corresponds to a constant sectional curvature metric with scalar curvature $\tau=\frac{3}{2 \kappa ^2} \left(\kappa ^4 \Lambda ^2-4 c_1^2-4
	c_2^2\right)$.
\item We consider the Anti-de Sitter space with coordinates $(x,y,z)$ and metric
\[
g_{AdS}= \kappa^2 \left(-\cosh^2 y dx^2+ dy^2+ \sinh^2 y dz^2\right).
\]
The scalar curvature is given by $\tau=-\frac{6}{\kappa^2}$. Functions of the form $h(x,y,z)=\frac{\kappa ^2 \Lambda }{2}+\sinh (y) (c_1 \cos (z)+c_2
\sin (z))$ provide solutions to \eqref{eq:vacuum-Einstein-field-equations} for constants $c_1$ and $c_2$. Note that the gradient of $h$ is always spacelike, since $\|\nh\|^2=\frac{1}{\kappa ^2}\left(\cosh ^2(y) \left(c_2 \sin (z)+c_1 \cos (z)\right){}^2\right.$ $\left.+\left(c_2 \cos (z)-c_1 \sin (z)\right){}^2\right)>0$. Therefore, there are solutions with spacelike $\nh$ for arbitrary $\Lambda$.

Moreover, the conformal metric $h^{-2} g_{AdS}$ corresponds again to Anti-de Sitter space with negative scalar curvature $\tau=-\frac{3}{2 \kappa ^2} \left(\kappa ^4 \Lambda ^2+4 c_1^2+4
	c_2^2\right)$.
\end{enumerate}
\end{remark}

Now, we continue the analysis of isotropic solutions to the vacuum weighted Einstein field equation. As a consequence of Lemma~\ref{le:Ricci-nilpotent} we have that $\tau=0$, $\Delta h=0$ and $\Lambda=0$, so equation \eqref{eq:vacuum-Einstein-field-equations} reduces to
\begin{equation}\label{eq:qE-mu1-equation}
h\rho=\Hes_h.
\end{equation}
Notice that this equation is linear in the function $h$. A more general version of \eqref{eq:qE-mu1-equation} was considered in \cite{BV-GR-G-VR} for affine manifolds.

\medspace

{\it Proof of Theorem~\ref{th:main}.}
We keep working in the pseudo-orthonormal basis $\mathcal{B}$ where, as a consequence of Lemma~\ref{le:Ricci-nilpotent}, the Ricci operator acts as follows:
\[
\begin{array}{l}
\Ric(\nabla h)=\Ric(E_i)=0, \text{ for } i=2,\dots, n-2,\\
\noalign{\medskip} 
\Ric(U)=\nu \nh+\mu E_1,\qquad \Ric(E_1)=\mu \nh.
\end{array}
\]
We distinguish three cases: $\Ric$ is zero ($\mu=\nu=0$), $\Ric$ is $2$-step nilpotent  ($\nu\neq0$ and $\mu= 0$) and $\Ric$ is $3$-step nilpotent ($\mu\neq 0$).

If the manifold is Ricci-flat, $\mu=\nu=0$, then  equation \eqref{eq:qE-mu1-equation} reduces to $\Hes_h=0$. Hence $\nabla h$ is a parallel vector field, so the manifold is a Ricci-flat Brinkmann wave with parallel vector field $\nh$. This proves Theorem~\ref{th:main}~(1).

If $\nu\neq0$ and $\mu= 0$, then the Ricci operator and, by \eqref{eq:qE-mu1-equation}, the Hessian operator are $2$-step nilpotent. We have $\nabla_{\nh} \nh=\nabla_{E_i}\nh=0$ for all $i=1,\dots,n-2$, while $\nabla_U \nh=h\nu \nh$, so $\nh$ is a lightlike recurrent vector field and the manifold is a Brinkmann wave. Theorem~\ref{th:main}~(2) follows.

If $\mu\neq 0$, then the Ricci and the Hessian operator are $3$-step nilpotent. We already know, by Lemma~\ref{le:geodesic-vector-field}, that the lightlike vector field $\nh$ is geodesic. We analyze the optical scalars \eqref{eq:optical-scalars} for $\nh$. Because $\nh$ is a gradient, it is twist-free ($\omega^2=0$). Moreover, we check that
\[
\theta=\frac{1}{n-2}\nabla_iV^i=\frac{1}{n-2}\Delta h=0,
\]
as a consequence of Lemma~\ref{le:Ricci-nilpotent}. Since $\operatorname{hes}_h$ is nilpotent and $\theta=0$, $\nh$ is also shear-free:
\[
\sigma^2=||\Hes_h||^2-(n-2)\theta^2=0.
\]
Hence, $\nh$ is a lightlike geodesic vector field with vanishing optical scalars, so we conclude that $(M,g)$ is a Kundt spacetime. This proves Theorem~\ref{th:main}~(3).\qed

The Ricci tensor of a warped product of the form $N\times_f I$, where $N$ is $n$-dimensional and $I\subset \mathbb{R}$ is a real interval, is given by \cite{Oneill}:
	\[
	\rho(X,Y)=\rho^N(X,Y) -\frac{1}f \operatorname{Hes}_f (X,Y),\quad \rho(X,\partial_t)=0,\quad \rho(\partial_t,\partial_t)=-\Delta f f,
	\]
	where $X,Y$ are vector fields tangent to $N$, $t$ is a coordinate parameterizing $I$ by arc length, and $\rho^N$ is the Ricci tensor of $N$. 	 Necessary and sufficient conditions for a  warped product $N\times_f I$ to be Einstein follow:
	\begin{eqnarray}
		&&\rho^N -\frac{1}f \operatorname{Hes}_f =\lambda g^N,\label{eq:warped-Einstein-condition-base}\\ 
		&& -\Delta f= \lambda f,\label{eq:warped-Einstein-condition-fiber}
	\end{eqnarray}
	where $\lambda$ is constant.
	By replacing $\lambda$ in equation \eqref{eq:warped-Einstein-condition-base} one gets $f\rho^N - \operatorname{Hes}_f + \Delta f g^N =0$, which corresponds to equation \eqref{eq:vacuum-Einstein-field-equations} with $\Lambda=0$. Thus, for any Einstein warped product $N\times_f I$, the smooth metric measure space $(N,g^N, f\, dvol_g)$ is a solution of the vacuum weighted Einstein field equation \eqref{eq:vacuum-Einstein-field-equations} with $\Lambda=0$.
	
	As a consequence of the results in Section~\ref{sect:general-dim}, isotropic solutions to the vacuum weighted Einstein field equation satisfy $\Delta h=0$ and $\Lambda=0$. Hence we obtain the following consequence.
\begin{corollary}\label{cor:4-warped}
	A smooth metric measure space $(N,g,h\, dvol_g)$ with isotropic density $h$ is a solution to the vacuum weighted Einstein field equation \eqref{eq:vacuum-Einstein-field-equations} if and only if $N\times_h \mathbb{R}$ is Einstein. 	Furthermore, in this case $N\times_h \mathbb{R}$ is Ricci-flat.
\end{corollary}

\section{The vacuum Einstein field equation in dimension three}\label{sect:3-dim}


\subsection{$pp$-waves}
We begin this section by classifying solutions to the vacuum Einstein field equation with the underlying structure of a $pp$-wave. 

\begin{theorem}\label{th:3-dim-pp-waves}
	Let $(M,g)$ be a $3$-dimensional $pp$-wave. If $(M,g,h\, dvol_g)$ is a non-flat solution of \eqref{eq:vacuum-Einstein-field-equations}, then $\Lambda=0$ and one of the following possibilities holds:
	\begin{enumerate}
		\item $\nabla h$ is lightlike and $(M,g)$ is a plane wave which in local coordinates can be written as
		\[
		g(u,v,x)=dv\left(2du-\frac{\alpha''(v)}{\alpha(v)} x^2 dv\right)+dx^2
		\]
		where $h(u,v,x)=\alpha(v)$ is an arbitrary function with $\alpha''(v)\neq 0$.
		\item $\nabla h$ is spacelike and $(M,g)$ can be written in local coordinates as in \eqref{eq:3-dim-Brinkmann-wave} with
		\[
		F(v,x)=\frac{\left(\gamma_{1} \alpha(v)+2 \gamma_{0}(v) \gamma_{0}''(v)\right) \log(\gamma_{0}(v)+\gamma_{1} x)}{\gamma_{1}^2}-\frac{2 x
			\gamma_{0}''(v)}{\gamma_{1}}+\beta(v),
		\]
		where  $h(u,v,x)=\gamma_{1} x+\gamma_{0}(v)$, $\gamma_1\in\mathbb{R}\backslash\{0\}$, and $\gamma_0$, $\alpha$, $\beta$ are arbitrary functions such that $\gamma_{1} \alpha(v)+2 \gamma_{0}(v) \gamma_{0}''(v)\neq 0$.
	\end{enumerate}
\end{theorem}
\begin{proof}
	Since $(M,g)$ is a $pp$-wave, there exist local coordinates so that the metric is given by \eqref{eq:3-dim-Brinkmann-wave} where $F(u,v,x)=F(v,x)$. Thus, we compute the expression of $G^h$: 
	\[
	\begin{array}{l}
	G^h(\partial_u,\partial_u)=-\partial_u^2 h,\;G^h(\partial_x,\partial_x)= \Lambda+2 \partial_u\partial_v h-F \partial_u^2 h, \; G^h(\partial_u,\partial_x)=-\partial_u\partial_x h,\\
	\noalign{\medskip}
	G^h(\partial_v,\partial_v)= F \left(-F 
	\partial_u^2h+\partial_x^2h+2 \partial_u\partial_vh+\Lambda
	\right)+\frac{\partial_vF
		\partial_uh-\partial_xF \partial_xh-2 \partial_v^2h-h\partial_x^2F}2,\\
		\noalign{\medskip}
	G^h(\partial_v,\partial_x)=-\partial_v\partial_x h+\frac{\partial_x F \partial_u h}{2}, \; G^h(\partial_u,\partial_v)=\Lambda+\partial_x^2 h+\partial_u\partial_v h-F \partial_u^2 h.
	\end{array}
	\]
From $G^h(\partial_u,\partial_u)=G^h(\partial_u,\partial_x)=0$ we get that $h(u,v,x)=h_1(v)u+h_0(v,x)$. Now, from  $G^h(\partial_x,\partial_x)=\Lambda+2h_1'(v)=0$, we get that $h_1(v)=-\frac{\Lambda}2 v+k$ for a constant $k$. From $G^h(\partial_u,\partial_v)=\Lambda+h_1'(v)+\partial_x^2 h_0(v,x)=0$, the function $h$ reduces to the form $h(u,v,x)=\left(-\frac{\Lambda}2 v+k\right)u-\frac{\Lambda}{4}x^2+h_{01}(v)x+h_{00}(v)$. 

If we differentiate $G^h(\partial_v,\partial_x)=-h_{01}'(v)+\frac{1}4 \left(2k-v \Lambda\right) \partial_x F(v,x)=0$ with respect to $x$, we obtain $\frac{1}4 \left(2k-v \Lambda\right) \partial_x^2 F(v,x)=0$. If $\partial_x^2 F(v,x)=0 $ then the manifold is Ricci flat and, hence, flat. Therefore, we conclude that $\Lambda=k=0$ and $G^h(\partial_v,\partial_x)=-h_{01}'(v)=0$, so $h_{01}$ is indeed constant. The function $h$ reduces to $h(u,v,x)=h_{01}x+h_{00}(v)$, with $\nabla h=h_{00}'(v)\partial_u+h_{01} \partial_x$ and $\|\nabla h\|^2=h_{01}^2$. 

We analyze separately the isotropic case ($\nabla h$ is lightlike: $h_{01}=0$) and the non-isotropic case ($\nabla h$ is spacelike: $h_{01}\neq 0$). If $h_{01}=0$, then the only non-vanishing component of $G^h$ is  $G^h(\partial_v,\partial_v)=-h_{00}''(v)-\frac12 h_{00}(v) \partial_x^2F(v,x)$. From $G^h=0$ we obtain that $F(v,x)$ is a polynomial of degree two of the form $F(v,x)=-\frac{h_{00}''(v)}{h_{00}(v)}x^2+F_1(v)x+F_0(v)$. Therefore $g$ is a plane wave and $F$ can be further normalized so that  $F(v,x)=-\frac{h_{00}''(v)}{h_{00}(v)}x^2$ (see, for example, \cite{leistner}). This corresponds to Assertion~(1).

We assume now that $\nabla h$ is spacelike, i.e. $h_{01}\neq 0$. There is only one remaining nonzero term of $G^h$:
\[
G^h(\partial_v,\partial_v)=\frac{1}{2} \left(-\partial_x^2F(v,x) (h_{00}(v)+h_{01}
x)-h_{01} \partial_xF(v,x)-2 h_{00}''(v)\right).
\]
We solve $G^h(\partial_v,\partial_v)=0$ to obtain the form of $F$ in terms of $\gamma_0(v)=h_{00}(v)$ and $\gamma_1=h_{01}$ as given in Assertion~(2).
\end{proof}

\subsection{Brinkmann waves}

It was shown in Theorem~\ref{th:main} that Brinkmann waves play a role when the Ricci operator is $2$-step nilpotent. We now show that all $3$-dimensional isotropic solutions in this case are indeed plane waves.

\smallskip

\noindent{\it Proof of Theorem~\ref{th:3-dim}(1).} We assume that the Ricci operator is $2$-step nilpotent. By Theorem~\ref{th:main}, $(M,g)$ is a Brinkmann wave where $\nabla h$ is a recurrent vector field. In dimension three, the fact that the Ricci operator is $2$-step nilpotent ensures that the Brinkmann wave admits a parallel null vector field (see \cite{leistner}) and the manifold is a $pp$-wave. Now the result follows from Theorem~\ref{th:3-dim-pp-waves}~(1).
\qed

\begin{remark}\label{re:plane-waves}
	Notice that, as a consequence of Theorem~\ref{th:3-dim}~(1), for any function $h(v)$ with $h''(v)\neq 0$ there always exists a plane wave $(M,g_{pw})$ so that $(M,g_{pw},h\, dvol_{g_{pw}})$ is an isotropic solution to the vacuum weighted Einstein field equation \eqref{eq:vacuum-Einstein-field-equations}.
	
	Among plane waves metrics, given by expression \eqref{eq:3-dim-Brinkmann-wave} with $F(v,x)=\alpha(v) x^2$, there are two families that are locally homogeneous \cite{Eduardo-Peter-Stana}:
	\begin{enumerate}
		\item The family $\mathcal{P}_c$, defined by $F(v,x)=- \beta(v) x^2$ with $\beta'=c \beta^{3/2}$ for a constant $c$ and $\beta>0$.
		\item The family of Cahen-Wallach symmetric spaces $\mathcal{CW}_\varepsilon$, defined by $F(v,x)=\varepsilon x^2$.
	\end{enumerate} 

Since solutions in Theorem~\ref{th:3-dim}~(1) are of the form $ F(v,x)=-\frac{\alpha''(v)}{\alpha(v)}x^2$, we have the following:
\begin{enumerate}
	\item Metrics in \eqref{eq:3-dim-Brinkmann-wave} with $ F(v,x)=-\frac{4}{c^2 v^2}x^2$ belong to the family $\mathcal{P}_c$ and, for $h(u,v,x)=a_1 (c v)^{\frac{c-\sqrt{c^2+16}}{2 c}}+a_2 (c v)^{\frac{c+\sqrt{c^2+16}}{2 c}}$, are homogeneous solutions to the  vacuum weighted Einstein field equation \eqref{eq:vacuum-Einstein-field-equations}. These metrics show null singularities and are geodesically incomplete (we refer to \cite{Blau-OLoughlin} for details).
	\item For $h(u,v,x)=b_1 e^{v \sqrt{\varepsilon }}+b_2 e^{-v \sqrt{\varepsilon }}$, if $\varepsilon>0$, and for $h(u,v,x)=b_1 \cos \left(v
	\sqrt{-\varepsilon }\right)$ $+b_2 \sin \left(v \sqrt{-\varepsilon }\right)$, if $\varepsilon<0$, Cahen-Wallach spaces $\mathcal{CW}_\varepsilon$ are solutions to the  vacuum weighted Einstein field equation \eqref{eq:vacuum-Einstein-field-equations}. Moreover, these metrics are geodesically complete (see \cite{Blau-OLoughlin,Cahen-Leroy-Parker-Tricerri-Vanhecke}). Also, for appropriate $h>0$ one has $\Hes_h\neq 0$, so there exist global solutions to \eqref{eq:vacuum-Einstein-field-equations}.
\end{enumerate}
\end{remark}

\begin{remark}

We analyze isotropic solutions to the vacuum weighted Einstein field equation \eqref{eq:vacuum-Einstein-field-equations} with a Brinkmann wave as a background metric by considering local coordinates as in \eqref{eq:3-dim-Brinkmann-wave}. By Lemma \ref{le:Ricci-nilpotent}, we have $\Lambda=\tau=\Delta h=0$. The scalar curvature takes the form $\tau=\partial_u^2F(u,v,x)$, thus we obtain $F(u,v,x)=F_1(v,x)u+F_0(v,x)$. With this reduction, the only nonzero component of the square of the Ricci operator is $\Ric^2(\partial_v)=\frac{1}{4} \left(\partial_xF_1\right)^2 \partial_u$. A direct calculation shows $G^h(\partial_u,\partial_u)=-\partial_u^2 h(u,v,x)$ and $G^h(\partial_u,\partial_x)=-\partial_u\partial_x h(u,v,x)$ and, from $G^h(\partial_u,\partial_u)=G^h(\partial_u,\partial_x)=0$, we get that $h(u,v,x)=h_1(v)u+h_0(v,x)$. We differentiate the term $G^h(\partial_x,\partial_x)=- h_1(v)F_1(v,x)+2h_1'(v)$ with respect to $x$ to see that $h_1(v)\partial_x F_1(v,x)=0$. Hence $h_1=0$ or $\partial_x F_1(v,x)=0$.

If $h_1(v)=0$, then $h(u,v,x)=h_0(v,x)$ and $0=\|\nabla h\|^2=\left(\partial_xh_{0}(v,x)\right)^2$, so the density function reduces to $h(u,v,x)=h_{00}(v)>0$. Now, we compute $0=G^h(\partial_v,\partial_x)=\frac{1}{2}h_{00}(v)\partial_x F_1(v,x)$ to obtain that in any case $\partial_x F_1(v,x)=0$. This condition yields $F_1(v,x)=F_1(v)$ and $(M,g)$ is at most 2-step nilpotent. It now follows that the manifold is a $pp$-wave (see, for example, \cite{leistner}). Hence, from Theorem~\ref{th:3-dim-pp-waves}, we conclude the following:
\begin{quote}\it
	If $(M,g,h\,dvol_g)$ is an isotropic solution  to the vacuum weighted Einstein field equation with $(M,g)$ a $3$-dimensional Brinkmann wave, then $(M,g)$ is a plane wave as described in Theorem~\ref{th:3-dim} (1).
\end{quote}
Moreover, notice that none of the Kundt spacetimes in Theorem \ref{th:3-dim} (2) are Brinkmann waves, since they are isotropic solutions with 3-step nilpotent Ricci operator.

In the cases where $\nabla h$ is not lightlike, however, we observe a loss of rigidity in the underlying manifold. Indeed, there exist non-isotropic solutions which are Brinkmann waves but not $pp$-waves. The following example illustrates this fact. 
\end{remark}

\begin{example}\label{ex:Brinkmann-non-isotropic}\rm
Let $(M,g)$ be a Brinkmann wave with metric given by \eqref{eq:3-dim-Brinkmann-wave} where
\[
F(v,x)=\frac{\left(4 u v-x^2\right) \log (v x)+x^2}{2 v^2}.
\]
The Ricci operator is given by 
\[
\Ric(\partial_u)=0,\quad \Ric(\partial_v)=\frac{4 u v+2 x^2 \log (v x)+x^2}{4 v^2 x^2}\partial_u+\frac{1}{v x}\partial_x, \quad \Ric(\partial_x)=\frac{1}{v x}\partial_u,
\]
so it is $3$-step nilpotent and, thus, it is not a $pp$-wave. A straightforward calculation shows that, for $h(u,v,x)=vx$ and  $\Lambda=0$, $(M,g,h)$ is a solution of equation \eqref{eq:vacuum-Einstein-field-equations}. Moreover, $\nh=x \partial_u+v\partial_x$, so $\|\nabla h\|^2=v^2$ and $\nh$ is spacelike.
\end{example}

As a consequence of Lemma~\ref{le:Ricci-nilpotent}, all isotropic solutions to the vacuum weighted Einstein equation \eqref{eq:vacuum-Einstein-field-equations} have vanishing scalar curvature. However, this is not necessarily the case if $\nabla h$ is not lightlike, as the following examples of Brinkmann waves show.

\begin{example}\label{ex:sc-neq-0}\rm
	We consider $\kappa\neq 0$ and define the following examples:
	\begin{enumerate}
		\item For $\kappa>0$, let $g$ be a Brinkmann metric defined by
		\eqref{eq:3-dim-Brinkmann-wave} with
		\[
		F(u,v,x)=\frac{u^2 \kappa}{2}+\alpha(v) \left(u+2 \sqrt{\frac{2}{\kappa}} \operatorname{arctanh}\left(\tan \left(\frac{x
			\sqrt{\kappa}}{2\sqrt{2}}\right)\right)\right).
		\]
		Then the scalar curvature is $\tau=\kappa$ and the manifold satisfies equation \eqref{eq:vacuum-Einstein-field-equations} for $h(u,v,x)=\cos \left(x\sqrt{\frac{\kappa}2}\right)$ and $\Lambda=0$. Moreover, 
		\[
		\nabla h=-\sqrt{\frac{\kappa}2} \sin \left(x\sqrt{\frac{\kappa}2}\right)\partial_x \text{ and } \|\nabla h\|=\frac12 \kappa \sin^2\left(x \sqrt{\frac{\kappa}2}\right)>0,
		\] 
		so the vector field $\nabla h$ is spacelike, since $\nabla h\neq 0$.
		
		\item For $\kappa<0$, let $g$ be a Brinkmann metric defined by
		\eqref{eq:3-dim-Brinkmann-wave} with
		\[
		F(u,v,x)=\frac{u^2 \kappa}2+\sqrt{\frac{2}{-\kappa}} \alpha(v) e^{-\frac{x \sqrt{-\kappa}}{\sqrt{2}}}.
		\]
		Then the scalar curvature is $\tau=\kappa$ and the manifold satisfies equation \eqref{eq:vacuum-Einstein-field-equations} for $h(u,v,x)=e^{\sqrt{\frac{-\kappa}{2}}x}$  and $\Lambda=0$. Moreover, 
		\[
		\nabla h=\sqrt{\frac{-\kappa}{2}}e^{\sqrt{\frac{-\kappa}{2}}x}\partial_x \text{ and } \|\nabla h\|=-\frac12 \kappa e^{\sqrt{-2\kappa}x}>0,
		\] 
		so the vector field $\nabla h$ is globally defined and it is spacelike.
	\end{enumerate} 
	
In conclusion, any constant scalar curvature $\tau$ is realizable by a solution of the vacuum Einstein field equation \eqref{eq:vacuum-Einstein-field-equations} with vanishing cosmological constant and a Brinkmann wave as a background metric.
\end{example}

\subsection{Kundt spacetimes}
We consider a $3$-dimensional Kundt spacetime and work with a metric given in local coordinates as in \eqref{eq:3-dim-kundt}.

\begin{lemma}\label{le:3-kundt}
Let $(M,g)$ be a $3$-dimensional Kundt spacetime with lightlike geodesic and expansion-free vector field $V$. If $\Ric(V)=0$ and $\tau=0$ then there exist local coordinates $(u,v,x)$ such that 
$g$ is of the form given in \eqref{eq:3-dim-kundt} with
\begin{equation}\label{eq:F-W-kundt-3-dim}
\begin{array}{rcl}
	F(u,v,x)&=&\frac{u^2}{x^2}+ \gamma_1(v,x) u+\gamma_0(v,x),\\
	\noalign{\smallskip}
	W(u,v,x)&=&-\frac{2u}{x}.
\end{array}
\end{equation}
\end{lemma}
\begin{proof}
We consider the form of the metric given in \eqref{eq:3-dim-kundt}, where $V=\partial_u$. A direct calculation shows that
\[
\Ric(V)= \frac{1}{2}
\left(\partial_u^2F-\partial_uW^2+\partial_u\partial_xW-2 W \partial_u^2W\right)\partial_u + \frac12 \partial_u^2 W\partial_x.
\]
Hence, since $\Ric(V)=0$, we have that $\partial_u^2 W=0$, so $W(u,v,x)=\omega_1(v,x)u+\omega_0(v,x)$. Now, $\Ric(V)=\frac{1}{2} \left(\partial_u^2F+\partial_x\omega_1-\omega_1^2\right)\partial_u$ and $\tau=\partial_u^2F+2 \partial_x\omega_1-\frac{3}{2} \omega_1^2$.
From these relations we obtain that  $2\partial_x\omega_1-\omega_1^2=0$ and, solving this differential equation, we obtain $\omega_1(v,x)=-\frac{2}{x+\varphi(v)}$. Moreover, since  $\partial_u^2F=\omega_1^2-\partial_x\omega_1=\partial_x\omega_1$, we get that $F(u,v,x)=\frac{u^2}{(x+\varphi (v))^2}+ \gamma_1(v,x) u+\gamma_0(v,x)$.

Appropriate changes of coordinates allow us to simplify the form of the functions $F$ and $W$ as follows. We refer to \cite{chow2010kundt} for changes of coordinates of $3$-dimensional Kundt spacetimes with functions $F$ and $W$ which are polynomial of degrees $3$ and $2$, respectively, in the variable $u$; and to \cite{podolsky2009general} for changes of coordinates in a broader context. Firstly, by setting $( u,v, x )=(\tilde u, \tilde v, \tilde x+\varphi(\tilde v))$ one can write $	F(u,v,x)=\frac{u^2}{x^2}+ \gamma_1(v,x) u+\gamma_0(v,x)$ and $W(u,v,x)=-\frac{2u}{x}+\omega_0(v,x)$. Moreover, a new change of the form $(u,v, x )=(\tilde u+\psi(\tilde v,\tilde x), \tilde v, \tilde x)$ for  $\psi(\tilde v,\tilde x)$ solving the equation $\omega_0 + \omega_1 \psi + \partial_{\tilde x} \psi=0$ transforms $W$ into a function of the form given above.
\end{proof}

{\it Proof of Theorem~\ref{th:3-dim}(2).} Let $(M, g, h\, dvol_g)$ an isotropic solution of \eqref{eq:vacuum-Einstein-field-equations}. If the Ricci operator is $3$-step nilpotent then, by Theorem~\ref{th:main}, $(M,g)$ is a Kundt spacetime where $\nh$ is the distinguished null geodesic expansion-free vector field. Hence, there exist coordinates $(u,v,x)$ as in \eqref{eq:3-dim-kundt} with $\nh=\partial_u$. For a general function $h(u,v,x)$ we compute
\[
\nabla h (u,v,x)=\left(\left(\omega^2-F\right)\partial_uh -\omega\partial_xh
+\partial_vh\right) \partial_u+\partial_uh\, \partial_v+\left(\partial_xh-\omega\partial_uh
\right)\partial_x
\]
to see that $\nabla h=\partial_u$ if and only if $h(u,v,x)=v+\kappa$, where $\kappa$ is a constant. We normalize the variable $v$ and consider  $h(u,v,x)=v$. Now, based on Lemma~\ref{le:3-kundt}, we consider $F$ and $W$ given by expression \eqref{eq:F-W-kundt-3-dim}. A direct calculation of the tensor $G^h$ shows that the nonzero components, up to symmetries, are
\[
\begin{array}{rcl}
G^h(\partial_v,\partial_v)&=&-\frac{ u
	v x \partial_x\gamma_1(v,x)- v x
	\partial_x\gamma_0(v,x)+ v
	\gamma_0(v,x)+ u}{ x^2}\\
\noalign{\smallskip}
&& -\frac{v
\partial_x^2\gamma_0(v,x)+ u v \partial_x^2\gamma_1(v,x)+ \gamma_1(v,x)}{2},\\
\noalign{\medskip}
G^h(\partial_v,\partial_x)&=&\frac{1}{2}
v \partial_x\gamma_1(v,x)+\frac{1}{x}.
\end{array}
\]
From $G^h(\partial_v,\partial_x)=0$ we get that $\gamma_1(v,x)=\alpha_1(v)-\frac{2 \log (x)}{v}$. Now, simplifying and solving $G^h(\partial_v,\partial_v)=0$, we obtain for $\gamma_0$ the expression in Theorem~\ref{th:3-dim}~(2). This completes the proof of Theorem~\ref{th:3-dim}~(2).\qed

\begin{remark}\label{re:VSI}
	A spacetime is said to have {\it vanishing scalar invariants} (VSI) (respectively, {\it constant scalar invariants} (CSI))  if all polynomial scalar invariants constructed from the curvature tensor and its covariant derivatives are zero (respectively, constant). 
	
	Three-dimensional locally CSI spacetimes were classified in \cite{CSI-dim3}, showing that they are locally homogeneous or a Kundt spacetime. Metrics in Theorem~\ref{th:3-dim}~(2) are a subclass of VSI Kundt metrics (cf. \cite{CSI1}).
\end{remark}

\begin{remark}\label{re:CPE}
	In \cite{Besse}, it was shown that an $n$-dimensional compact Riemannian manifold which is critical for the Einstein-Hilbert functional, restricted to the space of metrics with constant scalar curvature and unit volume, satisfies the {\it Critical Point Equation (CPE):}  
	\[
	(f+1)\rho-\Hes_f+\left(\Delta f-\frac{\tau}{n}\right) g=0,
	\]
	for a certain function $f$. Since the scalar curvature is assumed to be constant, this is a divergence-free equation formally similar to equation \eqref{eq:vacuum-Einstein-field-equations}. Besse conjectured in \cite{Besse} that the only critical compact Riemannian manifolds are standard spheres. Since then,  a number of papers have provided positive results under some extra assumptions (see, for example, \cite{Hwang2003,Neto}). A similar analysis to the one performed in Sections~\ref{sect:general-dim} and \ref{sect:3-dim} leads to classification results for solutions of this equation in the isotropic case if translated to Lorentzian signature. Furthermore, examples of solutions to this equation can be found among Kundt spacetimes and $pp$-waves. Thus, for example, since $\Delta f=\tau=0$ for isotropic solutions, $3$-dimensional Cahen-Wallach symmetric spaces ($\mathcal{CW}_\varepsilon$) provide geodesically complete solutions to the CPE, which are not Einstein, for $f(u,v,x)=c_1 e^{v \sqrt{\epsilon }}+c_2 e^{-v \sqrt{\epsilon }}-1$, if $\varepsilon>0$, and for $f(u,v,x)=c_1 \cos \left(v \sqrt{-\epsilon }\right)+c_2 \sin \left(v \sqrt{-\epsilon }\right)-1$, if $\varepsilon<0$ (cf. Remark~\ref{re:plane-waves}).	
\end{remark}

\section{Some remarks on four-dimensional spacetimes}\label{sect:4-dim-examples}

In view of Theorem~\ref{th:main}, if an isotropic solution to equation ~\eqref{eq:vacuum-Einstein-field-equations} is Ricci flat, then $\Hes_h=0$, so $\nabla h$ is a parallel lightlike vector field and the spacetime is a Brinkmann wave. 
The Ricci tensor determines the curvature in dimension three, so Ricci-flat $3$-dimensional manifolds are necessarily flat. However, there are $4$-dimensional isotropic solutions which are Ricci-flat but not flat. The following result shows that all these spacetimes are indeed $pp$-waves.

\begin{theorem}\label{th:4-dim-Ricci-flat}
Let $(M,g,h\, dvol_g)$ be a 4-dimensional isotropic Ricci-flat solution of the vacuum weighted Einstein field equation. Then $(M,g)$ is a $pp$-wave.
\end{theorem}
\begin{proof}
If $(M,g,h\, dvol_g)$ is an isotropic solution of \eqref{eq:vacuum-Einstein-field-equations} then, from Lemma~\ref{le:Ricci-nilpotent}, we have $\Delta h=0$ and $\Lambda=0$. Since $\rho=0$, equation \eqref{eq:vacuum-Einstein-field-equations} implies $\Hes_h=0$.
For arbitrary vector fields $X$, $Y$, $Z$ we have
\begin{equation}\label{eq:cur-zero-terms-h}
R(X,Y,Z,\nabla h)=(\nabla_X \Hes_h)(Y,Z)-(\nabla_Y \Hes_h)(X,Z)=0.
\end{equation}
Let $\mathcal{B}=\{\nh, U, E_1,E_2\}$ be a pseudo-orthonormal basis such that $g(\nh,U)=g(E_i,E_i)=1$ for $i=1,2$. Hence $\nabla h^\perp=\spanned\{\nabla h,E_1,E_2\}$. Due to \eqref{eq:cur-zero-terms-h}, we have that $R(\nabla h,E_i)=0$. We check that  $R(E_1,E_2)=0$ by computing
\[
\begin{array}{c}
0=\rho(E_2,U)=R(E_2,U,U,\nh)+R(E_2,E_1,U,E_1)=R(E_1,E_2,E_1,U),\\
\noalign{\smallskip}
0=\rho(E_1,U)=R(E_1,U,U,\nh)+R(E_1,E_2,U,E_2)=-R(E_1,E_2,E_2,U),\\
\noalign{\smallskip}
0=\rho(E_1,E_1)=2R(E_1,U,E_1,\nh)+R(E_1,E_2,E_1,E_2)= R(E_1,E_2,E_1,E_2).
\end{array}
\]
Therefore, $(M,g)$ is a Brinkmann wave with parallel lightlike vector field $\nabla h$ such that $R(\nabla h^\perp,\nabla h^\perp)=0$, so $(M,g)$ is a $pp$-wave.
\end{proof}

\begin{remark}
	A $pp$-wave of any dimension is given in local coordinates by expression \eqref{eq:local-coord-kundt-xeral} with $\partial_uF=0$, $W_{x_i}=0$ and $g_{ij}=\delta_{ij}$. The only possibly  nonzero component of its Ricci tensor is $\rho(\partial_v,\partial_v)=-\frac1{2}\bar{\Delta}F$, where $\bar{\Delta}=\sum_{i}\frac{\partial^2}{\partial x_i^2}$ is the Laplacian with respect to the flat spatial metric given by $g_{ij}$. Hence, a $pp$-wave is Ricci-flat if and only if $\bar{\Delta}F=0$. In dimension four, as a consequence of Theorem \ref{th:4-dim-Ricci-flat}, the only Ricci-flat isotropic solutions of the vacuum weighted Einstein field equation are $pp$-waves of this type.
	
	On the other hand, setting $h(u,v,x)=v$ in a $pp$-wave of arbitrary dimension, a straightforward calculation shows that $\nabla h=\partial_u$ is lightlike and $\Hes_h=0$. Thus, any $pp$-wave with $\bar{\Delta}F=0$ is a Ricci-flat isotropic solution of the vacuum weighted Einstein field equation with $h(u,v,x)=v$. 
\end{remark}

A natural question that arises in view of Theorem~\ref{th:3-dim} is  whether an analogous of assertion (1) holds in higher dimension. The  following example shows that, in general, isotropic solutions in Brinkmann waves to equation \eqref{eq:vacuum-Einstein-field-equations} do not need to be $pp$-waves, even if the Ricci operator is $2$-step nilpotent. 

\begin{example}\label{ex:Brinkmann-not-ppwave-4-dim}\rm
We consider local coordinates $(u,v,x_1,x_2)$ and the metric given, up to symmetry, by the following non-vanishing components:
\[
\begin{array}{l}
g(\partial_u,\partial_v)=1,\quad g(\partial_v,\partial_{x_2})=x_1x_2+vx_2^2,\\
\noalign{\medskip}
g(\partial_v,\partial_v)=   (-2 v x_2-x_1+2v x_2)u+\frac{-2 v^2 x_1^3 x_2-v x_1^4+3 v x_1^2 x_2^2+12 v x_1^2 x_2+x_1^3}{6 v}.
\end{array}
\]
The function $h(u,v,x_1,x_2)=v$ has lightlike gradient vector field $\nabla h=\partial_u$. A direct computation shows that this metric and the function $h$ provide a solution to the vacuum Einstein field equation \eqref{eq:vacuum-Einstein-field-equations} with $\Lambda=0$.  

The vector field $\nabla h$ is recurrent, since $\nabla\nabla h=-\frac{x_1}2 dv\otimes \nabla h$. Therefore, it is a Brinkmann wave. Moreover, the Ricci tensor has only one nonzero component: $\rho(\partial_v,\partial_v)=-\frac{x_1}{2v}$, so it is $2$-step nilpotent. 

Notice that $\nabla h^\perp=\spanned\{\partial_u,\partial_{x_1},\partial_{x_2}\}$. 
We check that 
\[
R(\partial_{x_1}, \partial_{x_2},\partial_v,\partial_{x_2})=\frac12,
\]
so $R(\nabla h^\perp,\nabla h^\perp)\neq 0$,  which means that the spacetime given by $g$ is not a $pp$-wave. Consequently, Theorem~\ref{th:3-dim}~(1) cannot be extended to higher dimension.
\end{example}

It was pointed out in Corollary~\ref{cor:4-warped} that isotropic solutions of the vacuum weighted Einstein field equation give rise to $4$-dimensional warped products which are Ricci-flat. The following are $4$-dimensional examples obtained by applying this construction.

\begin{example}\label{ex:4-dim-warpedproducts}\rm
We adopt notation from Theorem~\ref{th:3-dim-pp-waves}. Let $N_1$ be the plane wave given in Theorem~\ref{th:3-dim-pp-waves}~(1), let $h_1(u,v,x)=\alpha(v)$ and let $t$ be the coordinate of $\mathbb{R}$. The $4$-dimensional warped product $ M_1=N_1\times_{h_1} \mathbb{R}$ is Ricci-flat and its Weyl tensor (hence its curvature tensor) is determined, up to symmetries, by the following terms:
	\[
	W(\partial_v,\partial_x,\partial_v,\partial_x)=\frac{\alpha''(v)}{\alpha(v)} \,\,\text{ and }\,\, W(\partial_v,\partial_t,\partial_v,\partial_t)=-\alpha(v)\alpha''(v).
	\]
 Note that $M_1$ is still a Brinkmann wave with parallel lightlike vector field $V=\partial_u$. Furthermore, it satisfies the curvature conditions $R(V^\perp,V^\perp)=0$ and $\nabla_{V^\perp} R=0$, so it is indeed a plane wave.
	
Let $N_2$ be the $pp$-wave given in Theorem~\ref{th:3-dim-pp-waves}~(2) and $h_2(u,v,x)=\gamma_{1} x+\gamma_{0}(v)$. Then $M_2=N_2\times_{h_2} \mathbb{R}$ is a $4$-dimensional Ricci-flat warped product. Moreover, the Weyl tensor is determined, up to symmetries, by:
\[
W(\partial_v,\partial_x,\partial_v,\partial_x)=\frac{\gamma_1 \alpha(v)+2 \gamma_0(v) \gamma_0''(v)}{2 (\gamma_0(v)+\gamma_1 x)^2},  W(\partial_v,\partial_t,\partial_t,\partial_v)= \gamma_0(v) \gamma_0''(v)+\frac{\gamma_1 \alpha(v)}{2}.
\]
As in the previous example, $V=\partial_u$ is still parallel and $M_2$ satisfies $R(V^\perp,V^\perp)=0$, thus retaining the $pp$-wave character of $N_2$.
	
We adopt notation from Theorem~\ref{th:3-dim}~(2). Let $N_3$ be the Kundt spacetime given by \eqref{eq:Kundt-solutions} and $h_3(u,v,x)=v$. The $4$-dimensional warped product  $M_3=N_3\times_{h_3} \mathbb{R}$ is a Ricci-flat Kundt spacetime and its Weyl tensor is given, up to symmetries, by
\[
\begin{array}{c}
W(\partial_u,\partial_v,\partial_v,\partial_x)=-\frac{1}{vx},\quad W(\partial_v,\partial_t,\partial_v,\partial_t)=-\frac{1}{2} v \alpha_1(v)-\frac{u v}{x^2}+\log (x), \\
\noalign{\medskip} 
W(\partial_v,\partial_t,\partial_x,\partial_t)=\frac{v}{x}, \quad
W(\partial_v,\partial_x,\partial_v,\partial_x)=\frac{v \alpha_1(v)-\frac{6 u v}{x^2}-2 \log (x)}{2 v^2}.
\end{array}
\]
Since these examples are Ricci flat $4$-dimensional manifolds, they are solutions to the vaccum Einstein field equation. As such, their geometric information is encoded on the Weyl tensor, so it is convenient to analyze their Petrov type (we refer to \cite{Hall,stephani-et-al} for details). Since $M_1$ and $M_2$ are $pp$-waves, they are of type {\bf N} (one easily checks that $\iota_{\partial_u}W=0$). The warped product $M_3$, however, does not satisfy $\iota_{X}W=0$ for any vector field $X$, but $\iota_{\partial_u}W=-\frac{1}{vx}\,dv\otimes(dv\wedge dx)$, therefore it is of type {\bf III} (see \cite{Hall}). All these examples present a repeated principal null direction spanned by the distinguished lightlike vector field $\partial_u$. This is a common trait of Ricci-flat Kundt spacetimes, as a consequence of the Goldberg-Sachs theorem (see \cite{stephani-et-al}).

\end{example}

\section{Conclusions}
As a generalization of usual spacetimes, smooth metric measure spaces include a density function that affects  their geometry through the Bakry-\'Emery Ricci tensor. Based on this tensor, we propose a generalization of the Einstein tensor to this setting  as $	G^h=h \rho-\operatorname{Hes}_h +(\Delta h+\Lambda) g$ ({\it weighted Einstein tensor}), where  $\Lambda$ plays the role of a cosmological constant. $G^h$ preserves the main properties of being symmetric, concomitant of the metric $g$, the density function $h$
and their first two derivatives, and divergence-free (for manifolds with constant scalar curvature). Moreover, the expression of $G^h$ is related to the formal $L^2$-adjoint of the linearization of the scalar curvature function (Remark~\ref{re:static}).

The tensor $G^h$ gives rise to the vacuum weighted Einstein field equation $G^h=0$, whose solutions have constant scalar curvature (Lemma~\ref{le:const-sc}). We concentrate on the isotropic case, i.e. the case in which the gradient of $h$ is lightlike. Geometric conclusions are obtained and it is shown that isotropic solutions of the vacuum weighted Einstein field equation are: (i) Brinkmann waves with a parallel gradient vector field in the Ricci-flat case, (ii) Brinkmann waves if the Ricci operator is two-step nilpotent, and (iii) Kundt spacetimes if the Ricci operator is three-step nilpotent (see Theorem~\ref{th:main}).
Moreover, isotropic solutions to the vacuum weighted Einstein field equation are related  to Ricci-flat warped products with one-dimensional fiber (Corollary~\ref{cor:4-warped} and Remark~\ref{ex:4-dim-warpedproducts}).

More conclusive results are given in dimension three, where all non-flat isotropic solutions to the vacuum weighted Einstein field equation are described in local coordinates (Theorem~\ref{th:3-dim}). They are plane waves if the Ricci operator is two-step nilpotent and Kundt spacetimes with vanishing scalar invariants (VSI) if the Ricci operator is three-step nilpotent. Among plane waves, Cahen-Wallach symmetric spacetimes provide geodesically complete solutions (Remark~\ref{re:plane-waves}).

Several examples illustrate different phenomena for solutions in the non-isotropic case, where  the scalar curvature does not necessarily vanish (Remark~\ref{re:non-isotropic} and Examples~\ref{ex:Brinkmann-non-isotropic} and \ref{ex:sc-neq-0}), and in dimension  greater than three, where there exist isotropic solutions on Brinkmann waves that are not $pp$-waves (Example~\ref{ex:Brinkmann-not-ppwave-4-dim})).


\begin{thebibliography}{99}

\bibitem{Bakry-Emery}
D. Bakry and M. \'Emery; Diffusions Hypercontractives, in: S\'eminaire de probabilit\'es XIX,
1983/84, 177–206, Lecture Notes in Math. 1123, Springer, Berlin, 1985.

\bibitem{Besse}
A. L. Besse;
{\it Einstein manifolds.} Springer-Verlag (1987). 

\bibitem{Blau-OLoughlin}
M. Blau and M. O’Loughlin; Homogeneous plane waves. {\it Nucl. Phys. B} {\bf 654} (2003), 135--176.

\bibitem{Bourguignon}
J. P. Bourguignon;
Une stratification de l'espace des structures riemanniennes. 
{\it Compos. Math.} {\bf 30} (1975), 1--41. 

\bibitem{Br-Co-Her}
J. Br\"annlund, A. Coley and S. Hervik; Supersymmetry, holonomy and Kundt spacetimes.
{\it Class. Quantum Grav.} {\bf 25}, No. 19 (2008), 195007. 

\bibitem{BV-GR-G-VR}
M. Brozos-V\'azquez, E. Garc\'ia-R\'io, P. Gilkey and X. Valle-Regueiro,
A natural linear equation in affine geometry: the affine quasi-Einstein equation. {\it Proc. Am. Math. Soc.} {\bf 146} (8) (2018), 3485--3497.

\bibitem{Isotropic_quasi-Einstein}
M. Brozos-V\'azquez, E. Garc\'ia-R\'io, and X. Valle-Regueiro;
Isotropic quasi-Einstein manifolds.
{\it Class. Quantum Grav.} {\bf 36}, No. 24 (2019),  245005. 


\bibitem{cahen-wallach}
M. Cahen and N. Wallach; Lorentzian symmetric spaces. {\it Bull. Am. Math. Soc.} {\bf 76} (1970), 585--591. 

\bibitem{Cahen-Leroy-Parker-Tricerri-Vanhecke}
M. Cahen, J. Leroy, M. Parker, F. Tricerri, and L. Vanhecke; Lorentz manifolds modelled on a Lorentz symmetric space.
{\it J. Geom. Phys.} {\bf 7} (1990), 571--581.


\bibitem{Case-Shu-Wei}
J. Case, Y.-J. Shu, and G. Wei;
Rigidity of quasi-Einstein metrics. {\it 
Differ. Geom. Appl.} {\bf 29}, No. 1 (2011), 93--100. 


\bibitem{Catino2012}
G. Catino;
Generalized quasi-Einstein manifolds with harmonic Weyl tensor.
\emph{Math. Z.} \textbf{271} (2012),751--756.

\bibitem{Catino2013}
G. Catino, C. Mantegazza, L. Mazzieri, and M. Rimoldi; Locally conformally flat quasi-Einstein manifolds. {\it J. Reine Angew. Math.} \textbf{675} (2013), 181--189.

\bibitem{Case}
J. S. Case;
Singularity theorems and the Lorentzian splitting theorem for the Bakry-Emery-Ricci tensor. \emph{J. Geom. Phys.} \textbf{60} (2010), 477--490.

\bibitem{chow2010kundt}
D. D. Chow, C. Pope, and E. Sezgin; Kundt spacetimes as solutions of topologically massive gravity. {\it Class. Quantum Grav.}, {\bf 27},  No. 10 (2010), p. 105002.



\bibitem{Kundt-spacetimes}
A. Coley, S. Hervik, G. Papadopoulos, and N. Pelavas;
Kundt spacetimes. {\it Class. Quantum Grav.} {\bf 26}, No. 10 (2009) 105016.

\bibitem{CSI1}
A. Coley, S. Hervik, and N. Pelavas;
On spacetimes with constant scalar invariants. {\it Class. Quantum Grav.} {\bf 23}, No. 9 (2006), 3053--3074. 

\bibitem{CSI-dim3}
A. Coley, S. Hervik, and N. Pelavas;
Lorentzian spacetimes with constant curvature invariants in three dimensions. {\it Class. Quantum Grav.} {\bf 25}, No. 2, (2008), 025008. 


\bibitem{Corvino-etal}
J. Corvino,  M. Eichmair, and P. Miao;
Deformation of scalar curvature and volume. 
{\it Math. Ann.} {\bf 357}, No. 2 (2013), 551--584. 


\bibitem{Manolo-Eduardo-Kupeli}
M. Fern\'andez-L\'opez, E. Garc\'ia-R\'io, and D. Kupeli;  A local analytic characterization of Schwarzschild metrics. 
{\it J. Geom. Phys.} {\bf 45}, No. 3-4 (2003), 309--322. 

\bibitem{Fischer-Marsden}
A. E. Fischer and J. E. Marsden;
Linearization stability of nonlinear partial differential equations. {\it Differ. Geom., Proc. Symp. Pure Math.} {\bf 27}, Part 2, Stanford 1973, (1975),  219--263. 

%
%

\bibitem{Eduardo-Peter-Stana}
E. Garc\'ia-R\'io, P. Gilkey, S. Nik\v cevi\'c;
Homogeneity of Lorentzian three-manifolds with recurrent curvature. {\it Math. Nachr.} {\bf 287}, No. 1, (2014), 32--47. 

%


\bibitem{Hall}
G. Hall;
\emph{Symmetries and curvature structure in General Relativity.} 
World Scientific Lecture Notes in Physics, \textbf{46}, 
World Scientific Pub. Co., Singapore, 2004.

\bibitem{Hwang2003}
S. Hwang; The critical point equation on a three dimensional compact manifold. {\it Proc. Amer.Math. Soc.} {\bf 131}
(2003), 3221–3230.

%

\bibitem{Kobayashi}
O. Kobayashi; A differential equation arising from scalar curvature function. {\it J. Math. Soc. Japan} {\bf 34}  (1982), 665--675. 


\bibitem{Kuhnel-Rademacher}
W. K\"uhnel, and H.-B. Rademacher; 
Conformal transformations of pseudo-Riemannian manifolds.
\emph{Recent developments in pseudo-Riemannian geometry. ESI Lect. Math. Phys., Eur. Math. Soc.} (2008), 261--298. 


\bibitem{Lafontaine}
J. Lafontaine; Sur la g\'eom\'etrie d’une g\'en\'eralisation de l’\'equation diff\'erentielle d’Obata. {\it
J. Math. Pures Appl.}, IX. S\'er. 62 (1983),  63--72. 

\bibitem{leistner}
Th. Leistner; Conformal holonomy of C-spaces, Ricci-flat, and Lorentzian manifolds. {\it Differential Geom. Appl.} {\bf 24}, no. 5, (2006), 458--478. 

%
%

\bibitem{Lott}
J. Lott; Some geometric properties of the Bakry–\'Emery–Ricci tensor. {\it Comment. Math. Helv.} {\bf 78} (2003), 865--883.

\bibitem{Lovelock}
D. Lovelock; The Einstein tensor and its generalizations. 
\emph{J. Math. Phys.} {\bf 12}, (1971), 498--501. 



\bibitem{Miao-Tam}
P. Miao, and L.-F. Tam;
On the volume functional of compact manifolds with boundary with constant scalar curvature. 
{\it Calc. Var. Partial Differ. Equ.} {\bf 36}, No. 2  (2009), 141--171. 

\bibitem{MBZ}
A. Meliani, M. Boucetta, A. Zeghib; Kundt Three Dimensional Left Invariant Spacetimes. arXiv:2203.06379.

\bibitem{Neto}
B. L. Neto;
A note on critical point metrics of the total scalar curvature functional. 
{\it J. Math. Anal. Appl.} {\bf 424}, No. 2 (2015), 1544-1548. 

\bibitem{Oneill}
B. O'Neill;
\emph{Semi-Riemannian geometry. With applications to relativity.} 
Pure and Applied Mathematics, \textbf{103}, 
Academic Press, Inc., New York, 1983.


\bibitem{Osgood-Stowe}
B. Osgood and D. Stowe;
The Schwarzian derivative and conformal mapping of Riemannian manifolds. {\it Duke Math. J.} {\bf 67}, No. 1  (1992), 57--99. 

\bibitem{podolsky2009general}
J. Podolsk\`y and M. \^Zofka; General Kundt spacetimes in higher dimensions. {\it Class. Quamtum Grav.} {\bf 26} 10, (2009), 105008.

%
%
%

\bibitem{stephani-et-al}
H. Stephani, D. Kramer, M. MacCallum, C. Hoenselaers and E. Herlt;
{\it Exact solutions of Einstein’s field equations}.
Cambridge Monographs on Mathematical Physics. Cambridge University Press (2009). 

\bibitem{Xu}
X. Xu;
On the existence and uniqueness of solutions of M\"obius equations. {\it Trans. Am. Math. Soc.} {\bf 337}, No. 2  (1993), 927--945. 

\bibitem{Woolgar}
E. Woolgar; Scalar-tensor gravitation and the Bakry-\'Emery-Ricci tensor. {\it Class. Quantum Grav.} {\bf 30} (2013), 085007 (8pp).

%

\bibitem{Wu}
H. Wu, On the de Rham decomposition theorem. {\it Illinois J. Math.} {\bf 8} (1964), 291--311.

\end{thebibliography}
\end{document}